\documentclass[12pt]{article}

\usepackage[top=1.25in,bottom=1.25in,left=1.5in,right=1.5in]{geometry}
                                      
\usepackage[dvipsnames]{xcolor}

\usepackage{tikz}

\usepackage{cite} 
\usepackage{setspace}
\usepackage[latin1]{inputenc}
\usepackage{amsmath}
\usepackage{amsthm}
\usepackage{subfig}
\usepackage{amsfonts}
\usepackage{amssymb}
\usepackage{graphicx} 
\usepackage{epsfig}
\usepackage{url}
\usepackage{todo}
\usepackage{appendix}
\usepackage{bbm}

\usepackage{tikz}
\usepackage[dvipsnames]{xcolor}
\usetikzlibrary{cd}
\usepackage{pgfplots}
\usepackage{microtype}

\newtheorem{theorem}{Theorem}

\newtheorem{proposition}[theorem]{Proposition}

\newtheorem{remark}{Remark}
\newtheorem{example}{Example}

\def\qed{ \rule{.08in}{.08in}}

\newcommand{\blue}{\color{blue}}

\newcommand{\R}{\mathbb{R}}

\newcommand{\cL}{{\mathcal{L}}}

\newcommand{\dpt}{\operatorname{depth}}

\newcommand{\xc}[1]{\vspace{.3cm}

\noindent {\em #1} }
\newcommand{\mab}[1]{\vspace{.31cm}

\noindent {\bf #1} }
\makeatletter
\def\blfootnote{\xdef\@thefnmark{}\@footnotetext}
\makeatother

\usepackage{tikz}
\usetikzlibrary{shapes,arrows,graphs,graphs.standard,shapes.misc}
\usetikzlibrary{matrix,decorations.pathreplacing, calc, positioning,fit}
\usetikzlibrary{shapes.geometric}
\usetikzlibrary{decorations.markings}

\newtheorem{definition}{Definition}

\date{}

\begin{document}

\title{A Sufficient Condition for the Super-linearization of Polynomial Systems}
\author{M.-A. Belabbas\thanks{M.-A. Belabbas is with the Electrical and Computer Engineering Department and the Coordinated Science Laboratory, University of Illinois, Urbana-Champaign. Email: \texttt{\{belabbas@illinois.edu.}} \, and \,   
Xudong Chen\thanks{X. Chen is with the Department of Electrical, Computer, and Energy Engineering, University of Colorado Boulder. Email: \texttt{xudong.chen@colorado.edu}.}
}

\maketitle

\begin{abstract}
\blfootnote{M.-A. Belabbas and X. Chen contributed equally to the manuscript in all categories.}
 We provide in this paper a sufficient condition for a polynomial dynamical system $\dot x(t) = f(x(t))$ to be super-linearizable, i.e., to be such that all its trajectories are linear projections of the trajectories of a linear dynamical system. The condition is expressed in terms of the hereby introduced weighted dependency graph $G$, whose nodes $v_i$ correspond to variables $x_i$ and edges $v_iv_j$ have weights $\frac{\partial f_j}{\partial x_i}$. We show that if the product of the edge weights along any cycle in $G$ is a constant, then the system is super-linearizable. The proof is constructive, and we provide an algorithm to obtain super-linearizations and illustrate it on an example.
\end{abstract}

\section{Introduction}

The idea of linearizing system dynamics via embeddings dates back  at least  to the works of Carleman~\cite{carleman1932application} and Koopman~\cite{koopman1931hamiltonian,kowalski1991nonlinear}. These embeddings are still actively studied a century later, and have found applications in nonlinear control~\cite{brockett1976volterra, brockett2014early}, and  data-driven methods in control~\cite{mauroy2020koopman, otto2021koopman}. 

We derive in this paper a sufficient condition under which a polynomial system can be globally linearized by embedding it into a higher, yet finite-dimensional vector  space. In particular, the contribution of this paper is to provide a generalized converse of the result established in~\cite{belabbas2022canonical}. We elaborate on this below. 

To proceed, we consider the following dynamical system: 
 \begin{equation}\label{eq:mainsys}
 \dot x = f(x)	
 \end{equation}
 where $x \in \R^n$. This system is said to admit a  {\em super-linearization} (see Definition~\ref{def:suplinear} below) if there exist $m\geq 0$ functions, called  observables, which when adjoined to the original system would permit its linearization. A typical example~\cite{brunton2016koopman} is the following two-dimensional system
 \begin{equation}\label{eq:ex01}
 \begin{cases}
 \dot x = -x+y^2 \\
 \dot y = -y	
 \end{cases}
 \end{equation}
 Adding the observable $w:=y^2$, whose total time derivative is given by $\dot w = 2y\dot y = -2y^2 =-2w$, we obtain the three-dimensional {\em linear} system:
  \begin{equation}\label{eq:ex2}
 \begin{cases}
 \dot x = -x+w  \\
 \dot y = -y\\
 \dot w =-2w.	
 \end{cases}
 \end{equation}
Observe that the variables on which the nonlinear part of the dynamics~\eqref{eq:ex01} depend (here, the variable $y$) evolve in a linear, autonomous (i.e., independent from $x$) manner. In a recent paper~\cite{belabbas2022canonical}, we showed that {\em if} a polynomial system  admits a  so-called balanced super-linearization with only {\em one}  visible observable~\cite{belabbasobs2022}, then there exists a linear change of variables under which the nonlinear part of the dynamics  depends solely on variables evolving linearly and autonomously. The  dynamics resulting from the change of variable are termed the {\em canonical form}~\cite{belabbas2022canonical} for the polynomial system (explicitly, the canonical form is given in~\eqref{eq:triangularcase} below), and its existence provides a {\em necessary} condition for the super-linearization of that special class of polynomial systems.

Conversely, we exhibit in this paper a {\em sufficient} condition for the super-linearization of general polynomial systems, {\em without} any restriction on the number of visible observables. 
In particular, the result of this paper, combined with the ones of~\cite{belabbas2022canonical}, provide a necessary and sufficient condition for a class of polynomial systems with a single visible observable to be super-linearizable.

The remainder of the paper is organized as follows: We describe the relevant terminology and notation at the end of this section. We present the main result in Section~\ref{sec:mainresult} and its proof in Section~\ref{sec:proof}. The paper ends with a summary and outlook. 

\xc{Terminology and notation used.} We let $G = (V, E)$ be a directed graph (possibly with self-loops), with $V$ the node set and $E$ the edge set. We use $e = v_i v_j$ to denote a directed edge of $G$ from node $v_i$ to node $v_j$ (if $v_i = v_j$, then $e$ is a self-loop). A {\em walk} is a sequence of nodes $w = v_{i_1}v_{i_2}\ldots v_{i_k}$ such that $v_{i_{\ell}}v_{i_{\ell+1}}$ is an edge of $G$ for each $\ell= 1,\ldots, k-1$. The {\em length} of a walk is the number of edges it traverses. A path is a walk which does not visit a node more than once. We call the {\em depth} of $G$ the length of the longest path in $G$.

For a dynamical system $\dot x(t) = f(x(t))$, we denote by $e^{tf}x_0$ the solution of the system at time $t$ with  initial state $x_0$. 
For a vector field $g:\R^n \to \R^n$ and a differentiable vector-valued function $p:\R^n \to \R^k$, we denote the {\em Lie derivative} of $p$ along $g$ by 
$\cL_{g} p : = \frac{\partial p}{\partial x} g$.

\section{Statement of the Result}\label{sec:mainresult}
We start by defining what it means for system~\eqref{eq:mainsys} to be super-linearizable. Let $m \geq 0$ be an integer, and  $\Pi:\R^{n+m} \to \R^n$ be the canonical projection onto the first $n$ variables, namely, we have for $z \in \R^{n+m}$ that
$\Pi(z)=(z_1,\ldots,z_n)$. 
We reproduce the following definition from~\cite{belabbasobs2022}:
 
\begin{definition}[Super-linearization]\label{def:suplinear}
The vector field $f:\R^n\to \R^n$ is {\em super-linearizable} to the   system $\dot z = Az+D$  with $A \in \R^{(n+m)\times(n+m)}$ and $D \in \R^{n+m}$ if there exists an injective map $p:\R^n \to \R^{m}$   so that for all $x_0 \in \R^n$, the following holds: 
\begin{equation}\label{eq:equivalence1}
\Pi\left(e^{t(Az+D)}z_0 \right) = e^{tf}x_0	\mbox{ with } z_0=(x_0,  p(x_0)).
\end{equation} 
We call the functions $p:\R^n \to \R^m$ the {\em observables}.
\end{definition}
The data of $A, D$ and $p$ is referred to as a super-linearization of $f$. 
We can express the relation~\eqref{eq:equivalence1} as the following commutative diagram
\begin{center}
\begin{tikzcd}[column sep=huge, row sep=huge]
  \R^n \arrow[r, "e^{tf}"] \arrow[d,  "(\mathrm{id}{,}\,\,  p)" left] & \R^{n} \\
  \R^{n+m} \arrow[r,  "e^{t(Az+D)}" below] & \R^{n+m} \arrow[u,  "\Pi" right]
  \end{tikzcd}
\end{center}

We next introduce the following notion: 

\begin{definition}[Weighted dependency graph]\label{def:ddg}
Let $f: \R^n \to \R^n$ be a differentiable vector field. The {\bf weighted dependency graph (WDG)}  $G=(V,E,\gamma)$ of $f$ is a weighted directed graph (with self-loop) on $n$ nodes $v_1,\ldots, v_n$. For every ordered pair $(v_i,v_j)$, we define the scalar function:
$$
\gamma_{ij}(x):= \frac{\partial f_j(x)}{\partial x_i}\,\, \mbox{ for }\,\, 1 \leq i,j \leq n.
$$
There is an edge $v_iv_j$ in $G$ if  $\gamma_{ij} \neq 0$, and its weight is  $\gamma_{ij}$.
\end{definition}
We illustrate the definition on the following example:

\begin{example}\label{exmp:xumama}\normalfont
Consider the following polynomial system: 
\begin{equation}\label{eq:ex1}
\begin{cases} 
	\dot x_1 = x_2 \\
        \dot x_2 = -x_1 \\
	\dot x_3 =  x_2^2\\
        \dot x_4 = x_3 + x_1x_2^2 \\
        \dot x_5 = -x_5 + x_3^2 + x_1^2 x_2.
	\end{cases}
\end{equation}
Its weighted dependency graph is depicted in Figure~\ref{fig:figex1}. \hfill\qed
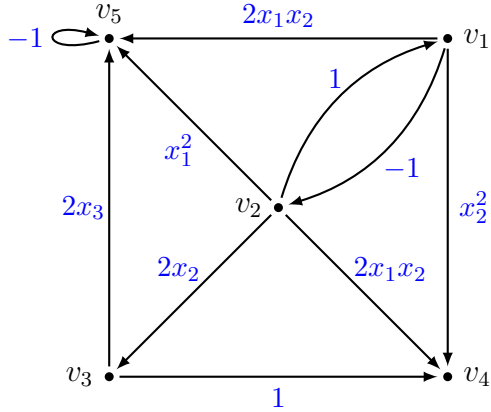
\begin{figure}
\begin{center}
\begin{tikzpicture}[scale=1.5]
	\node [circle,fill,inner sep=1.2pt,label=left:{$v_3$}] (2) at (0, 0) {};
	\node [circle,fill,inner sep=1.2pt,label=right:{$v_4$}] (3) at (3, 0) {};
	\node [circle,fill,inner sep=1.2pt,label=above:{$v_5$}] (1) at (0, 3) {};
        \node [circle,fill,inner sep=1.2pt,label=left:{$v_2$}] (5) at (1.5, 1.5) {};
        \node [circle,fill,inner sep=1.2pt,label=right:{$v_1$}] (4) at (3, 3) {};
  \path[draw,thick,shorten >=2pt,shorten <=2pt]
        (1) edge[loop left,min distance=6mm,-latex,in=165,out=-165] node[midway,left]{\small \blue{$-1$}}  (1) 
        (2) edge[-latex] node[midway,left,xshift=.1cm]{\small \blue{$2x_3$}}  (1) 
        (4) edge[-latex] node[midway,above,xshift=.0cm]{\small \blue{$2x_1x_2$}} (1)
        (5) edge[-latex] node[midway,below,xshift=-.2cm]{\small \blue{$x_1^2$}} (1)
        (5) edge[-latex] node[midway,above,xshift=-.2cm]{\small \blue{$2x_2$}} (2)
        (2) edge[-latex] node[midway,below,yshift=-0cm]{\small \blue{$1$}} (3)
        (4) edge[-latex] node[midway,right,xshift=.0cm]{\small \blue{$x_2^2$}} (3)
        (5) edge[-latex] node[midway,above,xshift=.361cm]{\small \blue{$2x_1x_2$}}(3)
	(4) edge[-latex, bend left = 28] node[midway,below right,xshift=-.2cm]{\small \blue{$-1$}}(5)
        (5) edge[-latex, bend left = 28] node[midway,above left,xshift=.2cm]{\small \blue{$1$}}(4)
	;
\end{tikzpicture}
\end{center}
\caption{The weighted dependency graph of system~\eqref{eq:ex1}. }\label{fig:figex1}
\end{figure}
\end{example}

Next, for each directed walk $w = v_{i_1}\ldots v_{i_k}$ in $G$, we let
$$
\gamma_w:= \prod_{j = 1}^{k-1} \gamma_{i_ji_{j+1}}.
$$
In the sequel, we will assume that $f$ is a polynomial vector field. It should be clear that $\gamma_w(x)$, for any walk $w$, is then a polynomial function in~$x$.  
Also, we assume, without loss of generality, that $G$ is weakly connected (otherwise, the original system can be decoupled into sub-systems of lower dimensions and our result, stated below, can be applied to each sub-system independently). 

The main result of this paper is as follows: 

\begin{theorem}\label{thm:main}
For a polynomial system $\dot x(t) = f(x(t))$, let $G$ be the associated weighted dependency graph.  
If $\gamma_c$ is a constant for every cycle $c$ of $G$, then~$f$ is super-linearizable.
\end{theorem}
The sufficient condition stated in the Theorem implies that the Jacobian of $f$ is constant.
Note that the weighted dependency graph of system~\eqref{eq:ex1}, depicted in Figure~\ref{fig:figex1}, satisfies the sufficient condition of Theorem~\ref{thm:main}, and thus system~\eqref{eq:ex1} is super-linearizable. As an illustration of the proof technique used, we will provide toward the end a super-linearization of this system.

\section{Proof of Theorem~\ref{thm:main} and an Algorithm}\label{sec:proof}

\subsection{Proof of the Theorem}
We  start with a simple proposition, dealing with systems where the variables on which the nonlinear part of the dynamics depend evolve linearly and autonomously, and show that such systems are super-linearizable.  This result provides a converse of the result of~\cite{belabbas2022canonical}, and will  be used as a building block to establish the general case.

\begin{proposition}\label{prop:propip}
Suppose that the system $\dot x(t) = f(x(t))$  takes the following form:
\begin{equation}\label{eq:triangularcase}
\begin{cases}
\dot x'(t) = A' x'(t) + D \\
\dot x''(t) = A'' x''(t) + g(x'(t)),
\end{cases}
\end{equation}
where $x = (x'; x'')$, $D$ is a constant vector, and $g$ is a polynomial; then, system~\eqref{eq:triangularcase} is super-linearizable.
\end{proposition}

\begin{proof}
Let $n'$ be the dimension of $x'$ and $d$ be the degree of $g$. Let $P_d$ be the vector space of all polynomials in~$x'$ with real coefficients, whose dimension is $$r:=\dim P_d = {n' + d\choose d}.$$ 
Next, for convenience, we let $f'(x'):= A'x' + D$. Since $f'$ is affine, $\cL_{f'} \phi\subseteq P_d$ for any $\phi\in P_d$ and, hence, $\cL_{f'}: P_d\to P_d$ is a linear automorphism. Let the minimal polynomial associated with $\cL_{f'}$ be given by
$$
s^N + \alpha_{N-1} s^{N-1} + \cdots + \alpha_0   
$$
for some $N\leq r$. In particular, for any $\phi\in P_d$, we have that
$$
(\cL_{f'}^N \phi) + \alpha_{N-1} (\cL_{f'}^{N-1} \phi)+ \cdots + \alpha_0 \phi = 0. 
$$
Now, define 
\begin{equation}\label{eq:defp} p(x) = \begin{bmatrix} p_1(x) \\ p_2(x) \\ \vdots \\ p_N(x) 
\end{bmatrix} :=\begin{bmatrix}
g(x') \\
\cL_{f'} g(x')\\
\vdots \\
\cL^{N}_{f'} g(x')
\end{bmatrix}.
\end{equation}
It then follows that the time derivative of $p(x(t))$ is
\begin{equation}\label{eq:defdynp}
\frac{d}{dt}
\begin{bmatrix}
p_1\\
p_2\\
\vdots \\
p_{N-1} \\
p_N
\end{bmatrix} = 
\begin{bmatrix}
0 & I & 0 & \cdots & 0\\
0 & 0 & I & \cdots  & 0\\
\vdots & \vdots & \ddots &  \ddots  & \vdots\\
0 & 0 & \cdots &  0 & I\\
-\alpha_0I & -\alpha_1I & \cdots & -\alpha_{N-2}I & -\alpha_{N-1}I
\end{bmatrix}
\begin{bmatrix}
p_1\\
p_2\\
\vdots \\
p_{N-1} \\
p_N
\end{bmatrix}.
\end{equation}
This completes the proof. 
\end{proof}

We next introduce two notions that are necessary for  enabling the recursive use of Proposition~\ref{prop:propip} in the proof of the main theorem.
The first is the notion of  strong component decomposition.

\begin{definition}[Strong component decomposition]\label{def:strongcomp} 
Let $G$ be a weakly connected digraph. 
The subgraphs $G_i = (V_i, E_i)$, for $1 \leq i \leq q$, form {\em a strong component
decomposition} of $G$ if the following items hold: 
\begin{enumerate}
\item The $V_i$'s partition the vertex set as $V = \sqcup_{i=1}^q V_i$;
\item Each $G_i$ is a subgraph induced by $V_i$ and is strongly connected;
\item Any strongly connected subgraph $G'$ of $G$ is a subgraph of some $G_i$, for $i\in \{1,\ldots, q\}$.
\end{enumerate}
\end{definition}

By treating the strongly connected components $G_i$ as single nodes, we obtain the second notion, namely the one of skeleton graph $S$ of $G$:

\begin{definition}\label{def:skeleton}
    Let $G=(V,E)$ be a weakly connected digraph, and let $G_1,\ldots, G_q$ be the strong component decomposition of $G$. The {\em skeleton graph $S = (U,F)$} is a digraph on $q$ nodes $u_1,\ldots, u_q$, corresponding to $G_1,\ldots, G_q$. There is no self-loop in $S$. There is an edge $u_i u_j$, for $u_i\neq u_j$, only if there exist a node $v_{i'}$ in $G_i$ and a node $v_{j'}$ in $G_j$ such that $v_{i'}v_{j'}$ is an edge in $G$. Further, we denote by $\pi:V \to U$ the map that sends nodes $v_{i'}$ in $V_i$ to $u_i$. 
\end{definition}
We illustrate the definition in Figure~\ref{fig:skelex1}.

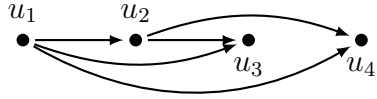
\begin{figure}
\begin{center}
\begin{tikzpicture}
   
\node [circle,fill=black,inner sep=1.7pt, label=above:{$u_1$}] (1) at (0, 0) {};
\node [circle,fill=black,inner sep=1.7pt,label=above:{$u_2$}] (2) at (1.5, 0) {};
\node [circle,fill=black,inner sep=1.7pt,label=below:{$u_3$}] (3) at (3, 0) {};
\node [circle,fill=black,inner sep=1.7pt,label=below:{$u_4$}] (4) at (4.5, 0) {};

\path[draw,thick,shorten >=2pt,shorten <=2pt]		 
          (1) edge[-latex] (2)
          (1) edge[-latex, bend right = 20] (3)
           (1) edge[-latex, bend right = 30] (4)
          (2) edge[-latex] (3)
		 (2) edge[-latex, bend left = 20] (4)
          ;
\end{tikzpicture}
\end{center}
\caption{The skeleton graph $S$ of the WDG $G$ of system~\eqref{eq:ex1}, depicted in Figure~\ref{fig:figex1}. Note that $\pi^{-1}(u_1)= \{v_1,v_2\}$, $\pi^{-1}(u_2)=\{v_3\}$, $\pi^{-1}(u_3)=\{v_4\}$, and $\pi^{-1}(u_4)=\{v_5\}$.
}\label{fig:skelex1}
\end{figure}

A subgraph $S' = (U',F')$ of $S$ induces a subgraph of $G$, obtained by only keeping the nodes of $G$ contained in the strong components represented by nodes of $S'$; precisely, to $S'$, we attach  the subgraph $G_{S'}$ of $G$ induced by $\pi^{-1}(U')$. See Figure~\ref{fig:apper} for an illustration. 
\begin{figure}
\begin{center}
    \begin{tikzpicture}[scale=0.8]

		\node [circle,fill=blue,inner sep=1.2pt] (11) at (0, 0) {};
		\node [circle,fill=blue,inner sep=1.2pt] (12) at (.665, .8) {};
		\node [circle,fill=blue,inner sep=1.2pt] (13) at (1.2, 0) {};

		\node [circle,fill=red,inner sep=1.2pt] (21) at (3, 0) {};
		\node [circle,fill=red,inner sep=1.2pt] (22) at (3, .8) {};

		\node [circle,fill=OliveGreen,inner sep=1.2pt] (31) at (1.4, -1.5) {};
		\node [circle,fill=OliveGreen,inner sep=1.2pt] (32) at (2.2, -1.5) {};

  \path[draw,thick,shorten >=2pt,shorten <=2pt]
		 (11) edge[blue,bend left=10,-latex] (12)
          (11) edge[loop left, -latex, blue, min distance=7mm,in=165,out=-165] (11) 
		 (12) edge[blue,bend left=10,-latex] (13)
		 (13) edge[blue,bend left=10,-latex] (11)
          (21) edge[red,bend left=20,-latex] (22)
		 (22) edge[red,bend left=20,-latex] (21)
          (31) edge[OliveGreen,bend left=20,-latex] (32)
		 (32) edge[OliveGreen,bend left=20,-latex] (31)
          (12) edge[-latex] (22)
          (13) edge[-latex] (22)
		 (11) edge[-latex] (31)
          (21) edge[-latex] (32)
		 ;
 \node at (1.8,-2.8){(a)};
\begin{scope}[xshift=7cm]

\node [circle,fill=blue,inner sep=1.7pt, label=above:{$u_1$}] (1) at (0, 0.4) {};
\node [circle,fill=red,inner sep=1.7pt,label=above:{$u_2$}] (2) at (2.2, 0.4) {};
\node [circle,fill=OliveGreen,inner sep=1.7pt,label=below:{$u_3$}] (3) at (1.1, -1.5) {};

\path[draw,thick,shorten >=2pt,shorten <=2pt]
		 
   (1) edge[-latex] (2)
		 (2) edge[-latex] (3)
		 (1) edge[-latex] (3)
          ;
\node at (1.1,-2.8){(b)};
\end{scope}

\begin{scope}[xshift=7cm,yshift=-4.9cm]

\node [circle,fill=blue,inner sep=1.7pt,label=above:{$u_1$}] (1') at (0, 0.4) {};
\node [circle,fill=red,inner sep=1.7pt,label=above:{$u_2$}] (2') at (2.2, 0.4) {};

\path[draw,thick,shorten >=2pt,shorten <=2pt]
		 (1') edge[-latex] (2');
\node at (1.1,-1.2){(c)};
\end{scope}

\begin{scope}[yshift=-4.9cm]
	\node [circle,fill=blue,inner sep=1.2pt] (11) at (0, 0) {};
		\node [circle,fill=blue,inner sep=1.2pt] (12) at (.665, .8) {};
		\node [circle,fill=blue,inner sep=1.2pt] (13) at (1.2, 0) {};

		\node [circle,fill=red,inner sep=1.2pt] (21) at (3, 0) {};
		\node [circle,fill=red,inner sep=1.2pt] (22) at (3, .8) {};

  \path[draw,thick,shorten >=2pt,shorten <=2pt]
		 (11) edge[blue,bend left=10,-latex] (12)
   (11) edge[loop left, -latex, blue, min distance=7mm,in=165,out=-165] (11)
		 (12) edge[blue,bend left=10,-latex] (13)
		 (13) edge[blue,bend left=10,-latex] (11)
          (21) edge[red,bend left=20,-latex] (22)
		 (22) edge[red,bend left=20,-latex] (21)
          (12) edge[-latex] (22)
          (13) edge[-latex] (22)
		 ;
    \node at (1.8,-1.2){(d)};
\end{scope}
\end{tikzpicture}
\end{center}
\caption{Illustration of $G_{S'}$: (a) A weakly connected digraph $G=(V,E)$, with three strongly connected components highlighted in blue, red, and green, respectively; (b) The skeleton graph $S=(U,F)$ of $G$, with $U=\{u_1,u_2,u_3\}$ and $F=\{u_1u_2,u_1u_3,u_2u_3\}$; (c) A subgraph $S'$ of $S$; and (d) The corresponding subgraph $G_{S'}$ of $G$.}\label{fig:apper}
\end{figure}
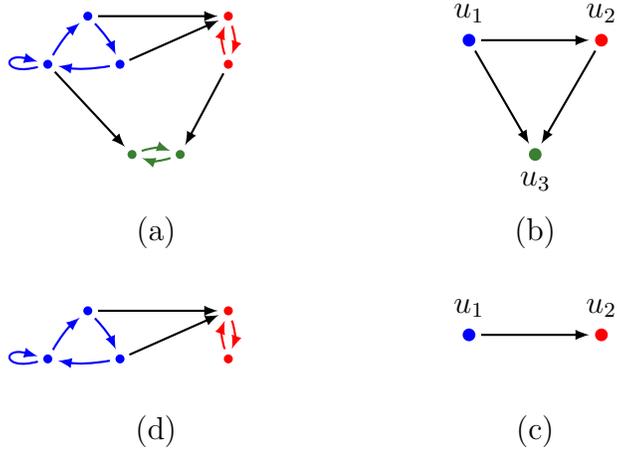
Note that the skeleton graph $S$ is acyclic because otherwise, it will contradict the third item of Definition~\ref{def:strongcomp}. 

Let $\ell$ be the depth of the graph $S$; we now introduce a node set decomposition, termed the {\em depth decomposition}, of $S$: 
\begin{equation}\label{eq:depthdecomp}
U = \sqcup_{m = 0}^\ell U_m.
\end{equation} 
Starting with $U_0$, we simply let it be the subset of nodes of $U$ {\em without} incoming edges. Since $S$ is acyclic, $U_0$ is non-empty. Now to each node $u_j$ in $U - U_0$, we assign the set $P_j$ of paths from nodes in $U_0$ to $u_j$. It should be clear that $P_j$ is non-empty. We define the {\em depth of the node} $u_j$, denoted by $\dpt (u_j)$,  to be the maximal length of all paths in $P_j$, i.e.,
$$
\dpt (u_j) := \max \{ \operatorname{length} (w) \mid w\in P_j\}.
$$
The subset $U_m$ is then the collection of all nodes in $S$ of depth~$m$. The subsets~$U_m$, for $0\leq m \leq \ell$, are all nonempty,  pairwise disjoint, and their union is~$U$.

With the preliminaries above, we establish Theorem~\ref{thm:main}

\begin{proof}[Proof of Theorem~\ref{thm:main}]
Let $G = (V, E, \gamma)$ be the weighted dependency graph of the polynomial vector field $f$. 
Let $S = (U,F)$ be the associated skeleton graph (obtained using Definition~\ref{def:skeleton} and ignoring  the weights~$\gamma$ of $G$), and $\ell$ be the depth of $S$. 
Because $G$ is weakly connected by assumption, so is $S$. The proof will be carried out by induction on~$\ell$.

\xc{Base case $\ell = 0$:} In this case, since $S$ is weakly connected, it is a single node. It follows  that $G$  is strongly connected. Next, we claim that {\em all} the weights $\gamma_{ij}$ for the edges $v_iv_j$ of $G$ are constant. To see this, for each edge $v_iv_j$ in $G$, we let $c=v_{i_1}v_{i_2}\cdots v_{i_k}v_{i_1}$ be a cycle in $G$ that contains this edge, with $v_{i_1} v_{i_2} = v_iv_j$.  By the hypothesis of Theorem~\ref{thm:main}, it holds that $\gamma_{c}$ is constant. We  have that
$$
\gamma_{v_iv_j} \gamma_{v_{i_2}\cdots v_{i_k}v_{i_1}}= \gamma_c. 
$$
Since $\gamma_c$ is a constant and since both $\gamma_{v_iv_j} $ and  $\gamma_{v_{i_2}\cdots v_{i_k}v_{i_1}}$ are  polynomials (over $\R$), it must hold that they are also constants. This establishes the claim. As a consequence, the vector field $f$ is an affine function. 
This completes the proof for the base case.

\xc{Inductive step:} We assume that the statement holds for $\ell\geq 0$ and prove it for $(\ell + 1)$. 
Let $\sqcup_{m = 0}^{\ell + 1} U_m$ be the node set decomposition of $U$  introduced in~\eqref{eq:depthdecomp}.  
Consider the subgraph~$S'$ of~$S$ induced by the nodes in $\sqcup_{m = 0}^\ell U_m$, and $S''$ the subgraph of $S$ induced by nodes in $U_{\ell + 1}$.

It should be clear that $S'$ is itself an acyclic digraph whose depth is~$\ell$,  and that $S''$ is a  union of isolated nodes. To see that the latter statement holds, it suffices to observe that if $S''$ has an edge, then it necessarily has nodes with different depths. 
We let $x'(t)$ be the vector with entries taken from  $x(t)$ corresponding to nodes in $G_{S'}$ and $x''(t)$ be the vector corresponding to $G_{S''}$. By construction of $S'$, the dynamics of $x'(t)$ do not depend on $x''(t)$ and, hence, we can write the said dynamics as $\dot x'(t) = f'(x'(t))$.  
On the one hand, by applying the induction hypothesis to each connected component of $S'$, we have that $f'$ is super-linearizable.    
We set $p'$ to be the associated observables, on which the super-linearization relies. 

On the other hand, the dynamics of $x''(t)$ may depend on both $x'(t)$ and $x''(t)$, i.e., $\dot x''(t) = f''(x'(t), x''(t))$ for $f''$ a polynomial vector field. Since each connected component of $G_{S''}$ is strongly connected, every edge in $G_{S''}$ belongs to a cycle in $G_{S''}$. By the hypothesis of Theorem~\ref{thm:main} and by the same arguments given in the base case, we then have that all the edge weights in $G_{S''}$ are constants. 
This implies that $f''(x',x'')$ is {\em affine} in $x''$ (note that edge weights in $G_{S''}$ only take into account differentiation of $f''$ with respect to $x''$, i.e., the variables corresponding to nodes $G_{S''}$). Combining the above, the dynamics can be expressed as
\begin{equation}\label{eq:laststep}
\begin{cases}
\dot z'(t) = A' z'(t) + D\\
\dot x''(t) = A''x''(t) + g(z'(t)),
\end{cases}
\end{equation}
where, owing to Proposition~\ref{prop:propip}, $z' := (x'; p')$, $A'$ and $A''$ are constant matrices, $D$ is a constant vector, and $g$ is a polynomial vector field. By Proposition~\ref{prop:propip}, system~\eqref{eq:laststep} is super-linearizable. This completes the proof.   
\end{proof}

\begin{remark}\label{rem:xumama}
\normalfont
Using  arguments similar to the ones of the proof of Theorem~\ref{thm:main}, one can establish the following fact (with proof omitted): suppose that $f$ is a smooth vector field; then, the system $\dot x = f(x)$ admits a super-linearization 
$\dot z =  A z + D$, 
where $z=\begin{pmatrix}x_1,\ldots,x_n, p_1,\ldots,p_m \end{pmatrix}$ as described below~\eqref{eq:laststep}, {\em if and only if} there exist an integer $N >0$ and coefficients $c_k \in \R$, for $k=0,\ldots, N-1$, such that
\begin{equation}\label{eq:XUMAMA}
\cL^N_f f = \sum_{k=0}^{N-1} \alpha_k \, \cL^{k}_f f.
\end{equation}
From that vantage point, the main result of this paper can  be restated as follows:  if~$f$ is a polynomial vector field and if~$f$ satisfies the condition of Theorem~1, then $f$ satisfies~\eqref{eq:XUMAMA} and is thus super-linearizable. 
\end{remark}

\subsection{Algorithm for Super-linearization}

The steps outlined in the proof of Theorem~\ref{thm:main} can be formalized as an algorithm, which we will present below. For ease of presentation, we introduce some notations.

Let $G$ be the WDG of a given system $\dot x = f(x)$ and $S$ be the corresponding skeleton graph. 
Let  $U = \sqcup_{m = 0}^\ell U_m$ be the depth decomposition, 
$S_m$ be the subgraph of $S$ induced by $U_m$. With a slight abuse of notation, we will use $x_{m}$ to denote the ``sub-vector'' of $x$ with entries corresponding to the nodes in $G_{S_{m}}$, 
and let $f_m(x)$ be defined such that $\dot x_m(t) = f_m(x(t))$.

The algorithm for super-linearization is as follows: 

\mab{Input:} A polynomial map $f:\R^n \to \R^n$ for the system $\dot x(t) = f(x(t))$.
\xc{Step 1:} Compute the WDG $G$ of the system and terminate if $G$ does not satisfy the  conditions of Theorem~\ref{thm:main}.  
\xc{Step 2:} Compute the skeleton graph $S = (U,F)$, its depth~$\ell$, and the depth decomposition $U = \sqcup_{m = 0}^\ell U_m$. 

\xc{Step 3:} Set $\ell' := 0$ and $z_{0}:= x_{0}$. While $\ell' < \ell$, repeat:
\begin{enumerate}
\item[\em 3.1:] Perform the super-linearization of the following system:
\begin{equation}\label{eq:presuperalg}
\begin{cases}
\dot z_{\ell'}(t) = A_{\ell'} z_{\ell'}(t) + D_{\ell'},\\
\dot x_{\ell'+1}(t) = f_{\ell'+1}(x(t)).
\end{cases}
\end{equation}
and obtain the super-linearized dynamics of~\eqref{eq:presuperalg}
\begin{equation}\label{eq:postsuperalg}
\dot z_{\ell' + 1}(t) = A_{\ell'+1} z_{\ell' + 1}(t) + D_{\ell' + 1}
\end{equation}
with observables $p_{\ell'+1}$.
\item[\em 3.2:] Increase $\ell'$ by $1$.
\end{enumerate}
\mab{Output:} The data $(A_\ell, D_\ell, p_\ell)$ as a super-linearization of the original system.

\begin{remark} \normalfont We elaborate below on a few points of Step 3.1 in the Algorithm: 
\begin{enumerate}
\item When $\ell' = 0$,~\eqref{eq:presuperalg} implies that the dynamics of $x_0$ are necessarily affine. It is indeed the case, and was argued in the proof of Theorem~\ref{thm:main} (the base case).

\item In~\eqref{eq:presuperalg}, the dynamics of $x_{\ell'+1}$ depend only on $x_{0}, \ldots, x_{\ell'+1} $ and, moreover, linearly in $x_{\ell'+1}$ as was argued in the  proof of Theorem~\ref{thm:main} (the inductive step). Note that $z_{\ell'+1}$ contains the variables $x_0,\ldots, x_{\ell'+1}$ and the observables $p_{\ell'+1}$.

\item In order to obtain the super-linearized dynamics~\eqref{eq:postsuperalg}, one can follow, e.g., the steps of the proof of Proposition~\ref{prop:propip}. The fact that~\eqref{eq:presuperalg} is in the same form as~\eqref{eq:triangularcase} is argued in the second item of this remark. More specifically, the first step is then  to determine the degree $d$ of the polynomial vector field $f_{\ell'+1}$. Next,  upon choosing a basis for $P_d$, determine the matrix of the linear operator $\cL_{ \bar f_{\ell'}}:P_d \to P_d$ where $\bar f_{\ell'}(z):= A_{\ell'} z + D_{\ell'}$ and compute the minimal polynomial of this matrix. Finally, introduce the observables $p$ as given in~\eqref{eq:defp}; they obey the linear dynamics~\eqref{eq:defdynp}. 

There exist other ways to obtain a super-linearization of the system; we will in fact follow a slightly different approach in the example next.
\end{enumerate}
\end{remark}

We illustrate the algorithm on the polynomial system given in Example~\ref{exmp:xumama}. Recall that the WDG $G$ of the system is given in Figure~\ref{fig:figex1}, and the corresponding skeleton graph $S = (U, F)$ is in Figure~\ref{fig:skelex1}. 

We next compute the depth decomposition of $U$. The only node that has no incoming edges is $u_1$, and thus $U_0 = \{u_1\}$. The longest path joining $u_1$ to $u_2$ is of length $1$, and the longest paths from $u_1$ to either $u_3$ or $u_4$ are of lengths $2$;  hence $U_1=\{u_2\}$ and $U_2=\{u_3,u_4\}$. 

Now, for Step 3, there will be two iterations: 
\begin{enumerate}
\item 
The first iteration considers the dynamics of the variables associated to $U_0$ (namely $x_1,x_2$) and $U_1$ (namely, $x_3$). We have
\begin{equation}\label{eq:algstep01}
\begin{cases} 
	\dot x_1 = x_2 \\
        \dot x_2 = -x_1 \\
	\dot x_3 =  x_2^2. \\
	\end{cases}
\end{equation}
We observe that the dynamics associated to the nodes in $U_0$ are indeed linear. 
Following~\eqref{eq:defp}, we set $x = (x',x'')$ with $x':=(x_1, x_2)$ and $x'':=x_3$, $p_1(x):=x_2^2$,  and $f'(x'):=(x_2, -x_1)$. We obtain that 
\begin{align*}
\cL_{f'} p_1 & = -2x_1x_2 =:p_2  \\
 \cL_{f'} p_2 & = 2(x_1^2-x_2^2)=:p_3\\
 \cL_{f'} p_3 & = 8x_1x_2 = -4p_2.
\end{align*}
The super-linearized system is thus
\begin{equation}\label{eq:iter1}
\dot z_1= 
\begin{bmatrix} 
\dot x_1\\
\dot x_2\\
\dot x_3\\
\dot p_1\\
\dot p_2\\
\dot p_3
\end{bmatrix}
= 
\begin{bmatrix} 
x_2\\
-x_1\\
p_1\\
p_2\\
p_3\\
-4p_2
\end{bmatrix} =: A_1z_1. 
\end{equation}
\item The second iteration starts with the super-linearized system~\eqref{eq:iter1} with the dynamics of the variables in $U_2$ adjoined. Namely, with 
$$
\begin{cases}
\dot z_1 = A_1 z_1\\
\dot x_4 = x_3+x_1x_2^2\\
\dot x_5 = -x_5 +x_3^2 + x_1^2x_2
\end{cases}
$$

To proceed, we could attempt to super-linearize the vector $(x_1x_2^2; x_3^2 + x_1^2x_2)$ at once, or handle each entry consecutively. We choose the latter option, which deviates slightly from the procedure described in Proposition~\ref{prop:propip} but requires fewer computations. Also, note that there is some freedom in how one expresses the nonlinear terms. For example, $x_1x_2^2$ can also be written as  $x_1p_1$ or $ -\frac{1}{2}x_2p_2$, given the observables introduced in the first iteration.  

We start by setting $p_4:=x_1x_2^2$ and $f'(z_1):=A_1z_1$. By computation, we obtain that
\begin{align*}
\cL_{f'} p_4 & =x_2^3 - 2x_1^2 x_2 =: p_5\\
\cL_{f'} p_5 & = -7x_1x_2^2 + 2x_1^3 =
 -7p_{4}+2x_1^3 =: p_6\\
\cL_{f'} p_6 & = -7p_5+6x_1^2x_2 =: p_7 \\
\cL_{f'} p_7 & = -7p_6 + 12x_1x_2^2-6x_1^3 \\
&= -7p_6+12p_4-3(p_6+7p_4)= -10p_6-9p_4.
\end{align*}

Next, we set $p_8:=x_3^2 + x_1^2x_2$ and
 \begin{align*}
\cL_{f'} p_8 & = 2x_3p_1 + 2x_1x_2^2-x_1^3  \\
&=2x_3p_1 -\frac{1}{2} (p_6+3p_4)=:p_9\\
\cL_{f'} p_9 & = 2p_1^2+2x_3p_2  - \frac{1}{2} (p_7+3p_5) =:p_{10}\\
\cL_{f'} p_{10} & = 6p_1p_2 + 2x_3 p_3 +\frac{1}{2}(9p_4 + 7p_6)=: p_{11}
\end{align*}
\begin{align*}
\cL_{f'} p_{11} & = 6p_2^2+8p_1p_3 - 8x_3 p_2 +\frac{1}{2}(9p_5 + 7p_7)=: p_{12} \\
\cL_{f'} p_{12} & = 20p_2p_3-40p_1p_2- 8x_3 p_3 -\frac{1}{2}(63p_4 + 61p_6)=: p_{13} \\
\cL_{f'} p_{13} & = 20p_3^2 - 120p_2^2 -48p_1p_3 + 32 x_3 p_2 -\frac{1}{2}(63p_5 + 61p_7)=: p_{14} \\
\cL_{f'} p_{14} & = -448p_2p_3 + 224 p_1p_2 + 32 x_3 p_3 +\frac{1}{2}(549p_4 + 547p_6)=:p_{15}\\
\cL_{f'} p_{15} & = 2016p_2^2 -448p_3^2  + 256 p_1p_3 - 128 x_3 p_2 +\frac{1}{2}(549p_5 + 547p_7) =:p_{16}\\
\cL_{f'} p_{16} & = 7872p_2p_3- 1152p_1p_2 - 128x_3p_3 - \frac{1}{2} (4923 p_4 + 4921 p_6) \\
& = \frac{1}{2}(1485p_4 + 1215p_6) - 256p_{11} - 144p_{13} - 24p_{15}.
\end{align*}

We thus obtain the following super-linearization of the original system~\eqref{eq:ex1}:
$$
\begin{cases}
\dot x_1 = x_2 \\
\dot x_2 = -x_1 \\
\dot x_3 = p_1 \\
\dot x_4 = x_3 + p_4 \\
\dot x_5 = -x_5 + p_7 \\
\dot p_i = p_{i+1}, \mbox{ for } i = 1,2,4,5,6,8,\cdots,15 \\
\dot p_3 = -4p_2 \\
\dot p_7 = -10p_6 - 9p_4 \\
\dot p_{16} = \frac{1485}{2}p_4 + \frac{1215}{2}p_6 - 256p_{11} - 144p_{13} - 24p_{15}.
\end{cases}
$$

\end{enumerate}

\section{Summary and Outlook}
We provided in this paper a sufficient condition for a system $\dot x(t) = f(x(t))$, with $f$ a polynomial vector field, to be super-linearizable. The condition is simply expressed in terms of cycles in what we called the weighted dependency graph of the system. The proof of the main result is constructive, and we have sketched an algorithm based on it that produces a super-linearization of  vector fields meeting the sufficient condition. The algorithm was also illustrated on an example.

The main result of this paper provides a generalized converse of the results in~\cite{belabbas2022canonical}. Indeed, while the canonical form exhibited there entails that in the original dynamics, the variables on which the nonlinear terms depend have to evolve linearly, it is easy to see that this fact does not hold for the system~\eqref{eq:ex1}. The gap of course lies in the fact that~\cite{belabbas2022canonical} restricts its scope to systems with only one visible observable, which precludes the nested super-linearizations that arise in the inductive step of the proof. In terms of the vocabulary introduced in this paper, the results of~\cite{belabbas2022canonical} only deal with skeleton graphs of {\em depth}~$1$. We will address the converse of the results presented in this paper, similarly generalize the results of~\cite{belabbas2022canonical}, in future work.

\bibliographystyle{IEEEtran}
\bibliography{carlerefs}

\end{document}